\documentclass[11pt]{amsart}
\usepackage{amsfonts}
\usepackage{graphicx}
\usepackage{amssymb}
\usepackage{amsmath}
\usepackage{amscd}
\usepackage{mathrsfs}
\usepackage{stmaryrd}
\usepackage{longtable}
\usepackage{txfonts}
\usepackage{wrapfig}
\usepackage{picins}
\usepackage{float}
\usepackage{xypic}
\usepackage{dsfont}
\usepackage{xcolor}
\usepackage{color}
\usepackage[colorlinks,linkcolor=blue,anchorcolor=blue,
citecolor=blue]{hyperref}
\usepackage{hyperref}
\allowdisplaybreaks

\newtheorem{theorem}{Theorem}[subsection]

\newtheorem{lem}[theorem]{Lemma}

\newtheorem{thm}[theorem]{Theorem}
\newtheorem{que}[theorem]{Question}

\begin{document}
\setlength{\oddsidemargin}{0cm}
\setlength{\evensidemargin}{0cm}

\title{Left-symmetric Structures on Complex Simple Lie Superalgebras}

\author{Runxuan Zhang}

\address{School of Mathematics and Statistics, Northeast Normal University,
 Changchun 130024, P.R. China}

\email{zhangrx728@nenu.edu.cn}

\date{\today}

\def\shorttitle{Left-symmetric Structures on Complex Simple Lie Superalgebras}

\begin{abstract}
A well-known fact is that there does not exist any compatible left-symmetric structures
 on a finite-dimensional complex semisimple Lie algebra (see \cite{Chu1974}). This result is not valid in semisimple Lie superalgebra  case.
In this paper, we study the compatible  Left-symmetric superalgebra (LSSA for short) structures  on complex simple Lie superalgebras.
We  prove that
there is not any compatible LSSA structure on  a finite-dimensional complex simple
Lie superalgebra except for the classical simple Lie superalgebra  $A(m,n)(m\neq n)$
and Cartan simple Lie superalgebra $W(n)(n\geq 3)$.  We also classify  all compatible LSSAs
with a right-identity on $A(0,1)$.
\end{abstract}



\maketitle
\baselineskip=16pt

\section{Introduction}
\setcounter{equation}{0}
\renewcommand{\theequation}
{1.\arabic{equation}}

\setcounter{theorem}{0}
\renewcommand{\thetheorem}
{1.\arabic{theorem}}

Let $k$ be a field. A nonassociative algebra $(\mathbb{A},\cdot)$ over $k$ is called a \textit{left-symmetric algebra} if $(x\cdot y)\cdot z-x\cdot(y\cdot z)=(y\cdot z)\cdot x-y\cdot(x\cdot z)$ for all $x,y,z\in \mathbb{A}$.
A superalgebra $(A=A_{\bar{0}}\oplus A_{\bar{1}},\cdot)$ is called a \textit{left-symmetric superalgebra} (\textit{LSSA} for short) if the associator $(x,y,z)=(x\cdot y)\cdot z-x\cdot(y\cdot z)$ is supersymmetric in $x,y$, i.e.,
$(x,y,z)=(-1)^{|x||y|}(y,x,z)$; or equivalently,  $$(x\cdot y)\cdot z-x\cdot(y\cdot z)=(-1)^{|x||y|}((y\cdot z)\cdot x-y\cdot(x\cdot z))$$ for all $x,y,z\in A$. We say that an LSSA $N$ is \textit{Novikov} if it satisfies an additional condition $(z\cdot x)\cdot y=(-1)^{|x||y|} (z\cdot y)\cdot x$ for all $x,y,z\in N$.

To calculate the cohomologies of associative algebras
and prove the  commutativity of cohomologies space, a kind of graded left-symmetric algebras were first introduced by Gerstenhaber \cite{Ger1963} in 1963.
Recently, LSSAs,  the super-version of left-symmetric algebras,  also appeared in several fields of mathematics and mathematical
physics as a natural algebraic structure.   In \cite{VM1996}, a linear basis of the free LSSA was constructed.
In 2008, Kong and Bai  classified  all compatible LSSA  on the super-Virasoro algebras satisfying some restricted conditions (\cite{KB2008}).
A close relation  between LSSAs and the graded classical Yang-Baxter equation in Lie superalgebras was obtained in \cite{WHB2010}.
As a subclass of LSSAs,  Novikov superalgebras were studied extensively in \cite{Xu2000a, Xu2000b, KC2009, AB2010, Zha2011,ZB2012,ZHB2011}.
Because LSSAs are Lie-admissible superalgebras,  a fundamental problem asks
whether there exist any compatible LSSAs on a give Lie superalgebra. The non-super version of this problem is of importance in geometry.
Actually, if $G$ is a connected and simply connected Lie group over the field of real numbers whose  Lie algebra is $\mathfrak{g}$, then  there is a left-invariant flat and torsion free connection, that is, an affine structure on $G$ if and only if $\mathfrak{g}$ has a compatible
left-symmetric algebra (\cite{Kos1961, Med1981}). There are a lot of papers addressing the compatible left-symmetric algebras on a given Lie algebra (see \cite{He1979, Kim1986, Bur1996, DIO1997, Bau1999}). In particular, an important result given by Chu \cite{Chu1974} asserts that  there do not exist any compatible left-symmetric algebra  on a complex finite-dimensional semisimple Lie algebra.

In view of pure mathematics, one can consider the question that asks whether there exist a compatible LSSA structure on a given complex finite-dimensional semisimple Lie superalgebra. We can construct a compatible LSSA structure on the Cartan complex simple Lie superalgebra $W(n)(n\geq 3)$ as follows: Let $\wedge(n)$ denote the Grassmann superalgebra with the
generators $\xi_{1},\cdots,\xi_{n}$, whose $\mathbf{Z}_{2}$-gradation is given
by $|\xi_{i}|=\bar{1}$ for $i=1,\cdots,n$. We write $W(n)$ for
$\textrm{Der}\wedge(n)$  and note that it is a Cartan simple Lie
superalgebra. Recall  that every
derivation $d\in W(n)$ can be written in the form
$$d=\sum_{i}u_{i}\frac{\partial}{\partial \xi_{i}},~~u_{i}\in
\wedge(n),$$ where $\frac{\partial}{\partial \xi_{i}}$ is the
derivation defined by $\frac{\partial}{\partial
\xi_{i}}(\xi_{j})=\delta_{ij}$ (see \cite{Kac1977}). We define the following
product on $W(n)$:
$$u\frac{\partial}{\partial \xi_{i}}\circ v\frac{\partial}{\partial \xi_{j}}=u\frac{\partial v}{\partial \xi_{i}}\frac{\partial}{\partial \xi_{j}},~~ u,v\in \wedge(n).$$
It is not difficult to check that the superalgebra $(W(n),\circ)$ is an LSSA.
This example indicates that the left-symmetric structures on a semisimple Lie superalgebra could be more complicated than the semisimple Lie algebra case.

In this paper, our attention is focused on the  compatible LSSA structures on complex simple Lie superalgebras.
In section 2, we prove that there is not any compatible LSSA structure on  a finite-dimensional complex simple
Lie superalgebra except for the classical simple Lie superalgebra  $A(m,n)(m\neq n)$
and Cartan simple Lie superalgebra $W(n)(n\geq 3)$.
Section 3 deals with the  compatible LSSAs on $A(0,1)$;
we construct a family of LSSAs $B_{2,\alpha}$ and two exceptional LSSA $B_{1},\widetilde{B_{2,-1}}$, and prove that
given a compatible LSSA $B$ on $A(0,1)$, if $B$ has a right-identity then it is isomorphic to one of $B_{1},\widetilde{B_{2,-1}}$, $B_{2,\alpha}$.

Throughout this paper, all (super)algebras are assumed to be finite-dimensional and over the field of complex numbers.

\section{Left-symmetric Superalgebras on Simple Lie Superalgebras}
\setcounter{equation}{0}
\renewcommand{\theequation}
{2.\arabic{equation}}

\setcounter{theorem}{0}
\renewcommand{\thetheorem}
{2.\arabic{theorem}}

Let us recall the  result due to Kac  (\cite{Kac1977}): all finite-dimensional
simple Lie superalgebras  that are not Lie algebras consist of the following classical and Cartan  Lie superalgebras
\begin{eqnarray*}
\textrm{classical}&:&A(m,n), n>m\geq 0; A(n,n),n\geq1; B(m,n),m\geq0,n\geq1; C(n),n\geq3;     \\&& D(m,n),m\geq2,n\geq1;   D(2,1;\alpha),\alpha\neq 0,-1; F(4);G(3);
P(n), n\geq2; \\&&  Q(n),n\geq2;\\
\textrm{Cartan}&:&W(n),n\geq3;  S(n),n\geq4;  \tilde{S}(n), n\geq 4, n \textrm{ even}; H(n),n\geq5.
\end{eqnarray*}
Note that
$D(2,1;\alpha)$ and  $D(2,1;\beta)$ are isomorphic  if and only if $\alpha$ and
$\beta$  lie in the same orbit of the group $V$ of order 6 generated by
$\alpha\mapsto -1-\alpha, \alpha\mapsto 1/\alpha$.

\begin{lem}\label{lem:2.1}
Let $\mathcal {G}=\mathcal {G}_{\bar{0}}\oplus\mathcal {G}_{\bar{1}}$ be a Lie superalgebra. If
 $\mathcal {G}_{\bar{0}}$  does not have a compatible left-symmetric algebra,
then there dose not exist a compatible LSSA  on $\mathcal {G}$.
\end{lem}

\begin{proof}
If $A$ is a compatible
LSSA on the Lie superalgebra $\mathcal {G}$, then $A_{\bar 0}$ is
a compatible left-symmetric algebra  on $\mathcal {G}_{\bar{0}}$, which is a contradiction.
\end{proof}

\begin{lem}[\cite{Chu1974}]\label{lem:2.2}
There does not exist any compatible left-symmetric algebra
 on a finite-dimensional semisimple Lie algebra over a field of
characteristic zero.
\end{lem}

\begin{lem}[\cite{Bur1996}]\label{lem:2.3}
 Let $\mathfrak{g}=\mathfrak{a}\oplus\mathfrak{z}$ be a finite-dimensional reductive Lie algebra over an algebraically closed field of
characteristic zero such that $\mathfrak{z}$ is a one-dimensional center
and  $\mathfrak{a}$ is a simple  ideal.  Then $\mathfrak{g}$ has a compatible  left-symmetric algebra if and only if $\mathfrak{a}$ is  of type  $A_{n}$.
\end{lem}

\begin{thm}
There does not exist any compatible LSSA on  a finite-dimensional simple Lie superalgebra except for
$A(m,n)(m\neq n)$ and $W(n)$.
\end{thm}

\begin{proof}
Let $\mathfrak{s}$ be the even part of the classical Lie superalgebras
$$A(n,n),B(m,n),D(m,n),D(2,1;\alpha),F(4),G(3),P(n),  Q(n).$$ Then $\mathfrak{s}=A_{n}\oplus A_{n},B_{m}\oplus
C_{n},D_{m}\oplus C_{n},A_{1}\oplus A_{1}\oplus A_{1},B_{3}\oplus
A_{1},G_{2}\oplus A_{1},A_{n}, A_{n}$ respectively. Since $\mathfrak{s}$ is a semisimple Lie algebra, it does not have a compatible left-symmetric algebra  by Lemma \ref{lem:2.2}.  It follows from Lemma \ref{lem:2.1} that  the above simple Lie superalgebras  do not have any compatible LSSAs.
On the other hand,  the even part
of $C(n)$ is isomorphic to the reductive Lie algebra $C_{n-1}\oplus \mathbf{C}$, then $C(n)_{\bar 0}$ does not have any compatible
left-symmetric algebra for $n\geq 3$ by Lemma \ref{lem:2.3}. Thus $C(n) (n\geq 3)$ does not have a compatible
LSSA structure, either.

For the $\textbf{Z}$-graded Lie superalgebras $S(n),\tilde{S}(n),H(n)$ of
Cartan type, their $\textbf{Z}$ degree zero parts $S(n)_{0}\cong
\mathfrak{sl}(n),\tilde{S}(n)_{0}\cong \mathfrak{sl}(n)$ and
$H(n)_{0}\cong \mathfrak{so}(n)$ are  simple Lie algebras and do not
have  any compatible left-symmetric algebras by Lemma \ref{lem:2.2}. Hence there do not exist any compatible
LSSAs on $S(n),\tilde{S}(n)$ or $H(n)$.
\end{proof}

\section{Left-symmetric Superalgebras on $A(0,1)$}

\setcounter{equation}{0}
\renewcommand{\theequation}
{3.\arabic{equation}}

\setcounter{theorem}{0}
\renewcommand{\thetheorem}
{3.\arabic{theorem}}

In this section, we deals with the left-symmetric superalgebra structures on $A(0,1)$. We need some preliminaries on compatible left-symmetric algebra structures on $\mathfrak{gl}(2)$.

\begin{lem}[\cite{Bau1999}]\label{lem:3.1}
Let $A$ be a compatible left-symmetric algebra on a reductive Lie algebra $\mathfrak{g}$, then  $A$ has a unique right-identity.
\end{lem}

Let  $x=\left(
\begin{smallmatrix}
0&1\\
0&0
\end{smallmatrix}
\right), y=\left(
\begin{smallmatrix}
0&0\\
1&0
\end{smallmatrix}\right), h=\left(
\begin{smallmatrix}
1&0\\
0&-1
\end{smallmatrix}
\right),z=\left(
\begin{smallmatrix}
1&0\\
0&1
\end{smallmatrix}
\right)$ be a basis of the Lie algebra $\mathfrak{gl}(2)$.
Using Lemma \ref{lem:3.1}, Burde classified all  compatible left-symmetric algebras on  $\mathfrak{gl}(2)$.

\begin{thm}[\cite{Bur1996}]\label{thm:3.2}
Let $A$ be a compatible left-symmetric algebra on $\mathfrak{gl}(2)$. Then it is isomorphic to $A_{1}, A_{2,\alpha}$  or $A_3$ defined by the left multiplication operators
$L(x), L(y),L(h),L(z)$ as follows:
\begin{enumerate}
 \item $\left(\begin{smallmatrix}
0&1/2&-1&1\\
0&0&0&0\\
0&1/2&0&0\\
0&1/2&0&0 \end{smallmatrix}
\right),
 \left(\begin{smallmatrix}
1/2&0&0&-1/2\\
0&0&1&1\\
-1/2&0&0&1/2\\
1/2&0&0&-1/2\end{smallmatrix}
\right),
\left(\begin{smallmatrix}
1&0&1&-1\\
0&-1&0&0\\
0&0&0&1\\
0&0&1&0\end{smallmatrix}
\right),
 \left(\begin{smallmatrix}
1&-1/2&-1&-1\\
0&1&0&0\\
0&1/2&1&0\\
0&-1/2&0&1\end{smallmatrix}
\right)$ ;\\
   \item  $
\left(\begin{smallmatrix}
0&0&-1&1+\alpha\\
0&0&0&0\\
0&(1+\alpha)/2&0&0\\
0&1/2&0&0 \end{smallmatrix}\right),
 \left(\begin{smallmatrix}
0&0&0&0\\
0&0&1&1-\alpha\\
(\alpha-1)/2&0&0&0\\
1/2&0&0&0\end{smallmatrix}
\right),
\left(\begin{smallmatrix}
1&0&0&0\\
0&-1&0&0\\
0&0&\alpha&1-\alpha^{2}\\
0&0&1&-\alpha\end{smallmatrix}
\right),
 \left(\begin{smallmatrix}
1+\alpha&0&0&0\\
0&1-\alpha&0&0\\
0&0&1-\alpha^{2}&\alpha^{3}-\alpha\\
0&0&-\alpha&1+\alpha^{2}\end{smallmatrix}
\right)$ ;\\
 \item $\left(\begin{smallmatrix}
0&0&-1&1\\
0&0&0&0\\
3&0&0&0\\
3&3&0&0 \end{smallmatrix}
\right),
 \left(\begin{smallmatrix}
0&0&1&0\\
0&0&-1&1\\
-1&-1/4&0&0\\
3&3/4&0&0\end{smallmatrix}
\right),
\left(\begin{smallmatrix}
1&1&0&0\\
0&-3&0&0\\
0&0&2&1\\
0&0&3&0\end{smallmatrix}
\right),
 \left(\begin{smallmatrix}
1&0&0&0\\
0&1&0&0\\
0&0&1&0\\
0&0&0&1\end{smallmatrix}
\right)$ ,
 \end{enumerate}
where $\alpha\in\mathbf{C}$. Two left-symmetric algebras $A_{2,\alpha}$ and $A_{2,\tilde{\alpha}}$ are isomorphic if and only if $\alpha^2=\tilde{\alpha}^2$. They are associative if and only if $\alpha=0$.
\end{thm}

\begin{lem}[\cite{Cha2004}]\label{lem:3.3}
Let $A$ be an LSSA. If the sub-adjacent Lie
superalgebra $\mathcal {G}_{A}$ is  simple, then $A$ is  simple as an LSSA.
\end{lem}

We denote by $e_{ij}$ the $3\times 3$ matrix having 1 in the $(i,j)$ position and 0
elsewhere. Let
\begin{eqnarray}
&&x_{1}=e_{23},x_{2}=e_{32},x_{3}=e_{22}-e_{33},x_{4}=2e_{11}+e_{22}+e_{33};\label{eq:3.1}\\
&&y_{1}=e_{12},y_{2}=e_{13},y_{3}=e_{21},y_{4}=e_{31}\label{eq:3.2}
\end{eqnarray}
be a  basis of the simple Lie superalgebra $A(0,1)$.  Then  the products are given as follows:
\begin{eqnarray*}
&&[x_{1},x_{2}]=x_{3},[x_{3},x_{1}]=2x_{1},[x_{3},x_{2}]=-2x_{2},[x_{1},x_{4}]=[x_{2},x_{4}]=[x_{3},x_{4}]=0,\\
&&[x_{1},y_{1}]=-y_{2},[x_{1},y_{2}]=[x_{1},y_{3}]=0, [x_{1},y_{4}]=y_{3},\\
&&[x_{2},y_{1}]=0, [x_{2},y_{2}]=-y_{1},[x_{2},y_{3}]=y_{4},[x_{2},y_{4}]=0,\\
&&[x_{3},y_{1}]=-y_{1},[x_{3},y_{2}]=y_{2},[x_{3},y_{3}]=y_{3},[x_{3},y_{4}]=-y_{4},\\
&&[x_{4},y_{1}]=y_{1},[x_{4},y_{2}]=y_{2},[x_{4},y_{3}]=-y_{3},[x_{4},y_{4}]=-y_{4},\\
&&[y_{1},y_{3}]=\frac{1}{2}(x_{3}+x_{4}),[y_{1},y_{4}]=x_{2},[y_{2},y_{3}]=x_{1},[y_{2},y_{4}]=\frac{1}{2}(x_{4}-x_{3}),\\
&&[y_{1},y_{1}]=[y_{1},y_{2}]=[y_{2},y_{2}]=[y_{3},y_{3}]=[y_{3},y_{4}]=[y_{4},y_{4}]=0.
\end{eqnarray*}

Now we consider the compatible LSSAs with a  right identity on $A(0,1)$. We would like to point out that a left-symmetric algebra has a right identity if and only if the sub-adjacent Lie algebra  admits an \'{e}tale
affine representation which leaves a point fixed (see \cite{Med1981}).

\begin{thm}\label{thm:3.4}
Let $B$ be a  compatible LSSA  on $A(0,1)$. Assume that  $B$ has a
right identity, then $B$ is isomorphism to  $B_{1},B_{2,\alpha}$ or
$\tilde{B}_{2,-1}$ defined by the left multiplication operators $L(x_i), L(y_j), i,j=1,2,3,4$, as follows:
\begin{enumerate}
  \item \begin{eqnarray*}
L(x_1)=\left(\begin{array}{c|cc}\begin{smallmatrix}
0&1/2&-1&1\\
0&0&0&0\\
0&1/2&0&0\\
0&1/2&0&0 \end{smallmatrix}&0\\
\hline
0&\begin{smallmatrix}0
\end{smallmatrix} \end{array}\right),
&& L(x_2)=\left(\begin{array}{c|cc}\begin{smallmatrix}
1/2&0&0&-1/2\\
0&0&1&1\\
-1/2&0&0&1/2\\
1/2&0&0&-1/2\end{smallmatrix}&0\\
\hline
0&\begin{smallmatrix}0
\end{smallmatrix} \end{array}\right),\\
L(x_3)=\left(\begin{array}{c|cc}\begin{smallmatrix}
1&0&1&-1\\
0&-1&0&0\\
0&0&0&1\\
0&0&1&0\end{smallmatrix}&0\\
\hline
0&\begin{smallmatrix}0
  \end{smallmatrix}\end{array}\right),
&& L(x_4)=\left(\begin{array}{c|cc}\begin{smallmatrix}
1&-1/2&-1&-1\\
0&1&0&0\\
0&1/2&1&0\\
0&-1/2&0&1\end{smallmatrix}&0\\
\hline
0& \begin{smallmatrix}
2&0&0&0\\
-1&2&0&0\\
0&0&0&1\\
0&0&0&0\\
 \end{smallmatrix}
 \end{array}\right),
\\
  L(y_{1})=\left(\begin{array}{c|c}
0&
\begin{smallmatrix}
0&0&0&1/4\\
0&0&0&0\\
0&0&0&-1/4\\
0&0&0&1/4\end{smallmatrix}\\
\hline\begin{smallmatrix}
0&0&1&1\\
1&0&0&-1\\
0&0&0&0\\
0&0&0&0\end{smallmatrix}&0
 \end{array}\right),
&&L(y_{2})=\left(\begin{array}{c|c}
0&
\begin{smallmatrix}
0&0&0&-1/2\\
0&0&0&0\\
0&0&0&0\\
0&0&0&0\end{smallmatrix}\\
\hline\begin{smallmatrix}
0&1&0&0\\
0&0&-1&1\\
0&0&0&0\\
0&0&0&0\end{smallmatrix}&0
 \end{array}\right),\\   L(y_{3})=\left(\begin{array}{c|c}
0&
\begin{smallmatrix}
0&1&0&0\\
0&0&0&0\\
1/2&0&0&0\\
1/2&0&0&0\end{smallmatrix}\\
\hline\begin{smallmatrix}
0&0&0&0\\
0&0&0&0\\
0&0&-1&1\\
0&-1&0&0\end{smallmatrix}&0
 \end{array}\right),
&&L(y_{4})=\left(\begin{array}{c|c}
0&
\begin{smallmatrix}
-1/4&1/2&0&0\\
1&0&0&0\\
1/4&-1/2&0&0\\
-1/4&1/2&0&0\end{smallmatrix}\\
\hline\begin{smallmatrix}
0&0&0&0\\
0&0&0&0\\
-1&0&0&1\\
0&0&1&1\end{smallmatrix}&0
 \end{array}\right); \end{eqnarray*}
\item  \begin{eqnarray*}
L(x_1)=\left(\begin{array}{c|cc}\begin{smallmatrix}
0&0&-1&\beta\\
0&0&0&0\\
0&\beta/2&0&0\\
0&1/2&0&0 \end{smallmatrix}&0\\
\hline
0&\begin{smallmatrix}0
\end{smallmatrix} \end{array}\right),
 &&L(x_2)=\left(\begin{array}{c|cc}\begin{smallmatrix}
0&0&0&0\\
0&0&1&\gamma\\
-\gamma/2&0&0&0\\
1/2&0&0&0\end{smallmatrix}&0\\
\hline
0&\begin{smallmatrix}0
\end{smallmatrix} \end{array}\right),\\
L(x_3)=\left(\begin{array}{c|cc}\begin{smallmatrix}
1&0&0&0\\
0&-1&0&0\\
0&0&\alpha&\beta\gamma\\
0&0&1&-\alpha\end{smallmatrix}&0\\
\hline
0&\begin{smallmatrix}0
\end{smallmatrix} \end{array}\right),
&& L(x_4)=\left(\begin{array}{c|cc}\begin{smallmatrix}
\beta&0&0&0\\
0&\gamma&0&0\\
0&0&\beta\gamma&-\alpha\beta\gamma\\
0&0&-\alpha&1+\alpha^{2}\end{smallmatrix}&0\\
\hline
0& \begin{smallmatrix}
2-\alpha&0&0&0\\
0&2+\alpha&0&0\\
0&0&\alpha&0\\
0&0&0&-\alpha\\
 \end{smallmatrix}
 \end{array}\right),\\
 L(y_{1})=\left(\begin{array}{c|c}
0&
\begin{smallmatrix}
0&0&0&0\\
0&0&0&\alpha\beta/4\\
0&0&-\alpha\gamma/4&0\\
0&0&\alpha/4&0\end{smallmatrix}\\
\hline\begin{smallmatrix}
0&0&1&\gamma\\
1&0&0&0\\
0&0&0&0\\
0&0&0&0\end{smallmatrix}&0
 \end{array}\right),
&&L(y_{2})=\left(\begin{array}{c|c}
0&
\begin{smallmatrix}
0&0&-\alpha/2&0\\
0&0&0&0\\
0&0&0&-\alpha\beta/4\\
0&0&0&-\alpha/4\end{smallmatrix}\\
\hline\begin{smallmatrix}
0&1&0&0\\
0&0&-1&\beta\\
0&0&0&0\\
0&0&0&0\end{smallmatrix}&0
 \end{array}\right),\\
L(y_{3})=\left(\begin{array}{c|c}
0&
\begin{smallmatrix}
0&1+\alpha/2&0&0\\
0&0&0&0\\
(\alpha\gamma+2)/4&0&0&0\\
(2-\alpha)/4&0&0&0\end{smallmatrix}\\
\hline\begin{smallmatrix}
0&0&0&0\\
0&0&0&0\\
0&0&-1&\beta\\
0&-1&0&0\end{smallmatrix}&0
 \end{array}\right),
&&L(y_{4})=\left(\begin{array}{c|c}
0&
\begin{smallmatrix}
0&0&0&0\\
1-\alpha\beta/4&0&0&0\\
0&(\alpha\beta-2)/4&0&0\\
0&(\alpha+2)/4&0&0\end{smallmatrix}\\
\hline\begin{smallmatrix}
0&0&0&0\\
0&0&0&0\\
-1&0&0&0\\
0&0&1&\gamma\end{smallmatrix}&0
 \end{array}\right); \end{eqnarray*}
 \item  \begin{eqnarray*}
L(x_1)=\left(\begin{array}{c|cc}\begin{smallmatrix}
0&0&-1&0\\
0&0&0&0\\
0&0&0&0\\
0&1/2&0&0 \end{smallmatrix}&0\\
\hline
0&\begin{smallmatrix}
0&0&0&0\\
1&0&0&0\\
0&1/2&0&1\\
-3/2&0&0&0\\
 \end{smallmatrix}
 \end{array}\right),
&& L(x_2)=\left(\begin{array}{c|cc}\begin{smallmatrix}
0&0&0&0\\
0&0&1&2\\
-1&0&0&0\\
1/2&0&0&0\end{smallmatrix}&0\\
\hline
0&\begin{smallmatrix}
0&-1&0&-2\\
0&0&-2&0\\
0&0&0&0\\
0&0&3&0\\
 \end{smallmatrix}
 \end{array}\right),\\
L(x_3)=\left(\begin{array}{c|cc}\begin{smallmatrix}
1&0&0&0\\
0&-1&0&0\\
0&0&-1&0\\
0&0&1&1\end{smallmatrix}&0\\
\hline
0&\begin{smallmatrix}
-2&0&0&0\\
0&0&0&0\\
0&0&2&0\\
0&0&0&0\end{smallmatrix}
 \end{array}\right),
&& L(x_4)=\left(\begin{array}{c|cc}\begin{smallmatrix}
0&0&0&0\\
0&2&0&0\\
0&0&0&0\\
0&0&1&2\end{smallmatrix}&0\\
\hline
0& \begin{smallmatrix}
1&0&0&0\\
0&1&0&0\\
0&0&1&0\\
0&0&0&1\\
 \end{smallmatrix}
 \end{array}\right),\\
L(y_{1})=\left(\begin{array}{c|c}
0&
\begin{smallmatrix}
0&0&0&0\\
0&3/4&0&1\\
0&0&0&0\\
0&0&3/4&0\end{smallmatrix}\\
\hline\begin{smallmatrix}
0&0&-1&0\\
2&0&0&0\\
0&0&0&0\\
-3/2&0&0&0\end{smallmatrix}&0
 \end{array}\right),
&&L(y_{2})=\left(\begin{array}{c|c}
0&
\begin{smallmatrix}
0&0&0&0\\
-3/4&0&0&0\\
0&0&0&0\\
0&0&0&1/4\end{smallmatrix}\\
\hline\begin{smallmatrix}
0&0&0&0\\
0&0&-1&0\\
1/2&0&0&0\\
0&0&0&0\end{smallmatrix}&0
 \end{array}\right),\\
L(y_{3})=\left(\begin{array}{c|c}
0&
\begin{smallmatrix}
0&1&0&1\\
0&0&0&0\\
1/2&0&0&0\\
-1/4&0&0&0\end{smallmatrix}\\
\hline\begin{smallmatrix}
0&0&0&0\\
0&-2&0&0\\
0&0&1&2\\
0&2&0&0\end{smallmatrix}&0
 \end{array}\right),
&&L(y_{4})=\left(\begin{array}{c|c}
0&
\begin{smallmatrix}
0&0&-1&0\\
0&0&0&0\\
0&-1/2&0&0\\
0&1/4&0&0\end{smallmatrix}\\
\hline\begin{smallmatrix}
0&-2&0&0\\
0&0&0&0\\
0&0&0&0\\
0&0&1&2\end{smallmatrix}&0
 \end{array}\right), \end{eqnarray*}
\end{enumerate}
where $\beta=1+\alpha, \gamma=1-\alpha, \alpha\in \mathbf{C}$ and $x_i,y_j, i,j=1,2,3,4$, are defined by  (\ref{eq:3.1}) and (\ref{eq:3.2}) respectively. They are simple and not associative.
\end{thm}

\begin{proof}

Let $B$ be a compatible LSSA on $A(0,1)$. Then $B_{\bar 0}$ is a compatible left-symmetric algebra on $A(0,1)_{\bar 0}\cong \mathfrak{gl}(2)$.
By Theorem \ref{thm:3.2}, $B_{\bar 0}$ is isomorphic to $A_{1}, A_{2,\alpha}$  or $A_{3}$.

Case (1):  $B_{\bar{0}}\cong A_{1}$. Since $L$  is an even linear map, the matrices $L(x_{i}), L(y_{j})(1\leq i,j\leq4)$ are of the form
\begin{eqnarray*}
&& L(x_1)=\left(\begin{array}{c|c}
\begin{smallmatrix}
0&1/2&-1&1\\
0&0&0&0\\
0&1/2&0&0\\
0&1/2&0&0 \end{smallmatrix}&0\\
\hline 0&(a_{ij})\\
 \end{array}\right),L(x_2)= \left(\begin{array}{c|c}
\begin{smallmatrix}
1/2&0&0&-1/2\\
0&0&1&1\\
-1/2&0&0&1/2\\
1/2&0&0&-1/2\end{smallmatrix}&0\\
\hline 0&(b_{ij})\\
 \end{array}\right),\\
 &&L(x_3)= \left(\begin{array}{c|c}
\begin{smallmatrix}
1&0&1&-1\\
0&-1&0&0\\
0&0&0&1\\
0&0&1&0\end{smallmatrix}&0\\
\hline 0&(c_{ij})\\
 \end{array}\right),L(x_4)= \left(\begin{array}{c|c}
\begin{smallmatrix}
1&-1/2&-1&-1\\
0&1&0&0\\
0&1/2&1&0\\
0&-1/2&0&1\end{smallmatrix}&0\\
\hline 0&(d_{ij})\\
 \end{array}\right), \\
 &&L(y_1)= \left(\begin{array}{c|c}
0&(e_{ij})\\
\hline (m_{ij})&0\\
 \end{array}\right),L(y_2)= \left(\begin{array}{c|c}
0&(f_{ij})\\
\hline (n_{ij})&0\\
 \end{array}\right),\\
 &&L(y_3)= \left(\begin{array}{c|c}
0&(g_{ij})\\
\hline (p_{ij})&0\\
 \end{array}\right),L(y_4)= \left(\begin{array}{c|c}
0&(h_{ij})\\
\hline (q_{ij})&0\\
 \end{array}\right),
 \end{eqnarray*}
where  $(a_{ij}),(b_{ij}), (c_{ij}), (d_{ij}), (e_{ij}), (f_{ij}), (g_{ij}), (h_{ij}), (m_{ij}), (n_{ij}), (p_{ij}), (q_{ij})$ are $4\times 4$  matrices.  Since $$[x_i,y_j]=x_iy_j-y_j x_i,$$ for all $ i,j=1,2,3,4,$ the matrices $(m_{ij}),(n_{ij}),(p_{ij})$ and $(q_{ij})$ can be determined by  $(a_{ij}),(b_{ij}),(c_{ij})$ and $(d_{ij})$. Namely,
\begin{eqnarray*}
&& (m_{ij})=
\begin{pmatrix}
a_{11}&b_{11}&c_{11}+1&d_{11}-1\\
a_{21}+1&b_{21}&c_{21}&d_{21}\\
a_{31}&b_{31}&c_{31}&d_{31}\\
a_{41}&b_{41}&c_{41}&d_{41} \end{pmatrix},
(n_{ij})=
\begin{pmatrix}
a_{12}&b_{12}+1&c_{12}&d_{12}\\
a_{22}&b_{22}&c_{22}-1&d_{22}-1\\
a_{32}&b_{32}&c_{32}&d_{32}\\
a_{42}&b_{42}&c_{42}&d_{42} \end{pmatrix},\\
 &&(p_{ij})=
\begin{pmatrix}
a_{13}&b_{13}&c_{13}&d_{13}\\
a_{23}&b_{23}&c_{23}&d_{23}\\
a_{33}&b_{33}&c_{33}-1&d_{33}+1\\
a_{43}&b_{43}-1&c_{43}&d_{43} \end{pmatrix},
 (q_{ij})=
\begin{pmatrix}
a_{14}&b_{14}&c_{14}&d_{14}\\
a_{24}&b_{24}&c_{24}&d_{24}\\
a_{34}-1&b_{34}&c_{34}&d_{34}\\
a_{44}&b_{44}&c_{44}+1&d_{44}+1
\end{pmatrix}.
 \end{eqnarray*}
Since $[y_i,y_j]=y_iy_j+y_j y_i$ and $[y_i,y_i]=0$ $(i,j=1,2,3,4)$, we have
\begin{eqnarray*}
(e_{ij})=
\begin{pmatrix}
0&e_{12}&e_{13}&e_{14}\\
0&e_{22}&e_{23}&e_{24}\\
0&b_{32}&e_{33}&e_{34}\\
0&b_{42}&e_{43}&e_{44} \end{pmatrix},
&& (f_{ij})=
\begin{pmatrix}
-e_{12}&0&f_{13}&f_{14}\\
-e_{22}&0&f_{23}&f_{24}\\
-e_{32}&0&f_{33}&f_{34}\\
-e_{42}&0&f_{43}&f_{44} \end{pmatrix}
,\\
 (g_{ij})=
\begin{pmatrix}
-e_{13}&-f_{13}+1&0&g_{14}\\
-e_{23}&-f_{23}&0&g_{24}\\
-e_{33}+1/2&-f_{33}&0&g_{34}\\
-e_{43}+1/2&-f_{43}&0&g_{44} \end{pmatrix},
&& (h_{ij})=
\begin{pmatrix}
-e_{14}&-f_{14}&-g_{14}&0\\
-e_{24}+1&-f_{24}&-g_{24}&0\\
-e_{34}&-f_{34}-1/2&-g_{34}&0\\
-e_{44}&-f_{44}+1/2&-g_{44}&0 \end{pmatrix}.
 \end{eqnarray*}
Note that $B$ has a  right identity $e$, that is,
 $xe=x$ for all $x\in B$. It is easy to see that  $e$ must be an even element. Suppose that $e=k_1x_1+k_2x_2+k_3x_3+k_4x_4$.
 From $$x_3=x_3e=k_1x_1-k_2x_2+k_3(x_1+x_4)+k_4(-x_1+x_3),$$ it follows that $k_1=1,k_2=0,k_3=0,k_4=1$, that is, $e=x_1+x_4$.
By Case (i) of Theorem \ref{thm:3.2}, we see that $e=x_1+x_4$ is the  right identity of $B_{\bar 0}.$ On the other hand, by the assumption that $e=x_1+x_4$
is the  right identity of $B$, we have
$y_i(x_1+x_4)=y_i, i=1,2,3,4$. Then $(d_{ij})$ can be determined by $(a_{ij})$, that is,
$$(d_{ij})=\begin{pmatrix}
2-a_{11}&-a_{12}&-a_{13}&-a_{14}\\
-1-a_{21}&2-a_{22}&-a_{23}&-a_{24}\\
-a_{31}&-a_{32}&-a_{33}&1-a_{34}\\
-a_{41}&-a_{42}&-a_{43}&-a_{44}
\end{pmatrix}.
$$

Notice that $(B,L)$ is a representation of  Lie superalgebra $A(0,1)$, we have $[L(x_1),L(x_4)]=0$ and thus $(a_{ij})(d_{ij})-(d_{ij})(a_{ij})=0$. A direct computation implies that
\begin{eqnarray}\label{eq:3.3}
&&a_{12}=a_{13}=a_{14}=a_{23}=a_{24}=a_{31}=a_{32}=a_{41}=a_{42}=a_{43}=0,\\
&&a_{11}=a_{22}, a_{33}=a_{44}.\nonumber
\end{eqnarray}
Since $[L(x_3),L(x_1)]=2L(x_1)$, it follows from (\ref{eq:3.3}) that
 \begin{eqnarray}
 &&a_{11}=a_{33}=c_{12}a_{21}=c_{43}a_{34}=0,\label{eq:3.4}\\
 &&a_{21}(2+c_{11}-c_{22})=a_{34}(2+c_{44}-c_{33})=0.\label{eq:3.5}
 \end{eqnarray}
Note that $[L(x_3),L(x_4)]=0$.  Combining (\ref{eq:3.3}) with  (\ref{eq:3.4}), we have
\begin{eqnarray}
 &&c_{13}=c_{14}=c_{23}=c_{24}=c_{31}=c_{32}=c_{41}=c_{42}=0,\label{eq:3.6}\\
 &&(1+a_{21})(c_{11}-c_{22})=(1-a_{34})(c_{33}-c_{44})=0.\label{eq:3.7}
 \end{eqnarray}
By the analogous computations,  $[L(x_2),L(x_4)]=0$, $[L(x_1),L(x_2)]=L(x_3)$, $[L(x_3),L(x_2)]=-2L(x_2)$ and (\ref{eq:3.3}), (\ref{eq:3.4}) and (\ref{eq:3.6}) give rise to the following conditions:
\begin{eqnarray*}
 &&b_{13}=b_{14}=b_{23}=b_{24}=b_{31}=b_{32}=b_{41}=b_{42}=0,\\
 &&(1+a_{21})(b_{11}-b_{22})=(1+a_{21})b_{12}=0,   c_{34}=a_{34}(b_{44}-b_{33}),  \\
 &&c_{12}=c_{43}=0, c_{22}=-c_{11}=a_{21}b_{12}, c_{33}=-c_{44}=a_{34}b_{43},\\
 &&c_{21}=a_{21}(b_{11}-b_{22}), (1-a_{34})(b_{33}-b_{44})=b_{43}(1-a_{34})=0,\\
 &&b_{11}=-b_{22}=\frac{1}{2}b_{12}c_{21}, b_{33}=-b_{44}=-\frac{1}{2}c_{34}b_{43},\\
 &&b_{21}(2+c_{22}-c_{11})=c_{21}(b_{22}-b_{11}), b_{34}(2+c_{33}-c_{44})=c_{34}(b_{33}-b_{44}).
 \end{eqnarray*}
It follows from (\ref{eq:3.5}) and (\ref{eq:3.7}) that $a_{21}=0$ or $-1$,  $a_{34}=0$ or $1$.  Hence there exist four possible cases.

Case (a):  $a_{ij}=b_{ij}=c_{ij}=d_{ij}=0$ for all $i,j=1,2,3,4$;

Case (b):  $a_{21}=-1, b_{12}=-1, c_{11}=-1, c_{22}=1, b_{11}=a, b_{22}=-a, b_{21}=a^2, c_{21}=-2a, a\in \mathbf{C}$ and  other $a_{ij}, b_{ij}, c_{ij}, d_{ij}$ are zero;

Case (c):  $a_{34}=1, b_{43}=1, c_{33}=1, c_{44}=-1, b_{33}=a, b_{44}=-a, b_{34}=-a^2, c_{34}=-2a, a\in \mathbf{C}$ and other $a_{ij}, b_{ij}, c_{ij}, d_{ij}$ are zero;

Case (d):  $a_{21}=-1, a_{34}=1, b_{12}=-1, b_{43}=1,c_{11}=-1, c_{22}=1, c_{33}=1, c_{44}=-1, b_{11}=a, b_{22}=-a, b_{21}=a^2, c_{21}=-2a,  b_{33}=b, b_{44}=-b, b_{34}=-b^2, c_{34}=-2b, a,b\in \mathbf{C}$
and other $a_{ij}, b_{ij}, c_{ij}, d_{ij}$ are zero.

 Since $[L(x_1),L(y_2)]=0$, we have $$a_{21}(b_{12}+1)+\frac{1}{2}c_{11}=0.$$ Then the above Cases (b) and (d) do not exist.
  By $[L(x_1),L(y_3)]=0$, it follows that $$a_{34}(b_{43}-1)-\frac{1}{2}c_{33}=0.$$ Then the above Case (c) does not exist.
Moreover, by $[L(y_i),L(y_i)]=0, i=1,2,3,4, [L(y_1),L(y_2)]=0$ and Case (a), we have  a unique solution $$e_{14}=e_{44}=-e_{34}=\frac{1}{4}, f_{14}=-\frac{1}{2}   \textrm{ and  other } e_{ij}, f_{ij}, g_{ij} \textrm{ are zero}.$$
Hence we obtain the superalgebra $B_1$ given by Case (i).
We observe that the supercommutator of $B_1$  gives  Lie superalgebra $A(0,1)$ and   $(B_1,L)$ is a representation of $A(0,1)$.
 Therefore,  $B_1$ is a compatible LSSA on $A(0,1)$.

 Case (2): $B_{\bar{0}}\cong A_{2,\alpha}$.
 Suppose that $e=k_1x_1+k_2x_2+k_3x_3+k_4x_4$ is a right identity.
By $$x_1=x_1e=k_2(\frac{1+\alpha}{2}x_3+\frac{1}{2}x_4)-k_3x_1+k_4(1+\alpha) x_1$$
and
$$x_2=x_2e=k_1(-\frac{1-\alpha}{2}x_3+\frac{1}{2}x_4)+k_3x_2+k_4(1-\alpha) x_2,$$ we have $k_1=k_2=0,k_3=\alpha,k_4=1$, that is, $e=\alpha x_3+x_4$.
By the analogous computations, we obtain the superalgebra
 $B_{2,\alpha}$ with a  right identity $e=\alpha x_3+x_4$. When $\alpha=-1$, we obtain an exceptional superalgebra
 $\widetilde{B_{2,-1}}$ with a  right identity $e=-x_3+x_4$..

 Case (3):    $B_{\bar{0}}\cong A_{3}$.
 It follows
 from the straightforward calculations that there is not a compatible
 LSSA  on $A(0,1)$ having a  right identity in this case.

By Lemma \ref{lem:3.3}, it is obvious that $B_{1}, B_{2,\alpha}$ and
$\widetilde{B_{2,-1}}$ are simple. By Theorem \ref{thm:3.2}, the even parts of $B_{1}, B_{2,\alpha}(\alpha\neq0)$ and
$\widetilde{B_{2,-1}}$ are not associative, so these LSSAs are not associative. Note that $B_{2,0}$ is not
 associative  since $(x_{3} x_{3})y_{1}-x_{3}(x_{3}
 y_{1})=2y_{1}\neq 0$.  This completes the proof.
\end{proof}

\begin{lem}\label{lem:3.5}
$B_{2,\alpha_1}\simeq B_{2,\alpha_2}$ if and only if $\alpha_1=\alpha_2$.
\end{lem}

\begin{proof}
Assume that $B_{2,\alpha_1}$ and $B_{2,\alpha_2}$ are isomorphic. Then their even parts are isomorphic.
By Theorem \ref{thm:3.2}, we have $\alpha_1=\pm\alpha_2$. If $B_{2,\alpha}$ and $B_{2,-\alpha}(\alpha\neq0)$ are isomorphic, then there exists an isomorphism $\varphi: B_{2,\alpha}\rightarrow B_{2,-\alpha}$. Let $P=(p_{ij})$ be the matrix of $\varphi$.  Since $\varphi$  is an even linear map, $P$ is of the form
$$\left(\begin{array}{c|cc}\begin{smallmatrix}
p_{11}&p_{12}&p_{13}&p_{14}\\
p_{21}&p_{22}&p_{23}&p_{24}\\
p_{31}&p_{32}&p_{33}&p_{34}\\
p_{41}&p_{42}&p_{43}&p_{44}\end{smallmatrix}&0\\
\hline
0& \begin{smallmatrix}
p_{55}&p_{56}&p_{57}&p_{58}\\
p_{65}&p_{66}&p_{67}&p_{68}\\
p_{75}&p_{76}&p_{77}&p_{78}\\
p_{85}&p_{86}&p_{87}&p_{88}\\
 \end{smallmatrix}
 \end{array}\right).$$
As $\varphi$ is an isomorphism, the determinant of $P$ cannot be zero.
Since $\varphi(x_1x_3)=\varphi(x_1)\varphi(x_3)$ and $\varphi(x_3x_1)=\varphi(x_3)\varphi(x_1)$, we have
\begin{equation}\label{eq:3.8}
\left\{\begin{aligned}
&-p_{11}p_{33}+(1-\alpha)p_{11}p_{43}+p_{31}p_{13}+(1-\alpha)p_{13}p_{41}=-p_{11},\\
&p_{21}p_{33}+(1+\alpha)p_{21}p_{43}-p_{31}p_{23}+(1+\alpha)p_{41}p_{23}=-p_{21},\\
&\frac{1-\alpha}{2}p_{11}p_{23}-\frac{1+\alpha}{2}p_{13}p_{21}+p_{31}((1-\alpha^2)p_{43}-\alpha p_{33})\\&+(1-\alpha^2)p_{41}(p_{33}+\alpha p_{43})=-p_{31},\\
&\frac{1}{2}p_{11}p_{23}+\frac{1}{2}p_{21}p_{13}+p_{31}p_{33}+\alpha p_{31}p_{43}+\alpha p_{41}p_{33}+(1+\alpha^2)p_{41}p_{43}=-p_{41}
 \end{aligned}\right.
 \end{equation}
and
\begin{equation}\label{eq:3.9}
\left\{\begin{aligned}
&p_{11}p_{33}+(1-\alpha)p_{11}p_{43}-p_{31}p_{13}+(1-\alpha)p_{13}p_{41}=p_{11},\\
&-p_{21}p_{33}+(1+\alpha)p_{21}p_{43}+p_{31}p_{23}+(1+\alpha)p_{41}p_{23}=p_{21},\\
&-\frac{1+\alpha}{2}p_{11}p_{23}+\frac{1-\alpha}{2}p_{13}p_{21}+p_{33}((1-\alpha^2)p_{41}-\alpha p_{31})\\&+(1-\alpha^2)p_{43}(p_{31}+\alpha p_{41})=p_{31},\\
&\frac{1}{2}p_{11}p_{23}+\frac{1}{2}p_{21}p_{13}+p_{31}p_{33}+\alpha p_{31}p_{43}+\alpha p_{41}p_{33}+(1+\alpha^2)p_{41}p_{43}=p_{41}.
 \end{aligned}\right.
 \end{equation}
Combining (\ref{eq:3.8}) with (\ref{eq:3.9}), we obtain
 \begin{equation}\label{eq:3.10}
\left\{\begin{aligned}
&p_{41}=0, (1-\alpha)p_{11}p_{43}=0, p_{11}p_{33}-p_{13}p_{31}=p_{11},\\
& (1+\alpha)p_{21}p_{43}=0, p_{31}p_{23}-p_{21}p_{33}=p_{21},2p_{31}=p_{21}p_{13}-p_{11}p_{23},\\
&-\alpha p_{11}p_{23}-\alpha p_{13}p_{21}-2\alpha p_{33} p_{31}+2(1-\alpha^2)p_{43}p_{31}=0,\\
&\frac{1}{2}p_{11}p_{23}+\frac{1}{2}p_{21}p_{13}+p_{31}p_{33}+\alpha p_{31}p_{43}=0.
 \end{aligned}\right.
 \end{equation}
With an analogous arguments for $\varphi(x_ix_j)=\varphi(x_i)\varphi(x_j), 1\leq i,j\leq 4.$
If $\alpha\neq 0$, then
$$p_{21}=k,p_{12}=\frac{1}{k}, p_{33}=-1, p_{44}=1,$$ and other  $p_{ij}(1\leq i,j\leq 4)$ are zero, where $0\neq k\in \mathbf{C}$.
Similarly,  $$p_{56}=kp_{65}\neq0, p_{87}=kp_{78}\neq0,  \textrm{ and other } p_{ij} (5\leq i,j\leq 8) \textrm{ are zero},$$ where $0\neq k\in \mathbf{C}$.
Observe that  $\varphi(y_1y_3)=\varphi(y_1)\varphi(y_3), \varphi(y_2y_3)=\varphi(y_2)\varphi(x_3)$ and $\varphi(y_1y_4)=\varphi(y_1)\varphi(y_4)$ give rise to the following conditions on the variables of $P$:
\begin{eqnarray}
&&kp_{65}p_{78}=1, \label{eq:3.11}\\
&&kp_{65}p_{78}=\frac{1+\alpha}{2},\label{eq:3.12}\\
&&kp_{65}p_{78}(1-\alpha)=2.\label{eq:3.13}
\end{eqnarray}
Substituting (\ref{eq:3.11}) into (\ref{eq:3.12}) and (\ref{eq:3.13}) respectively, we get $1+\alpha=2$ and $1-\alpha=2$, which is a contradiction.  Thus  $B_{2,\alpha}$ and $B_{2,-\alpha}$ are not isomorphic.
\end{proof}

Finally we need to prove

\begin{thm}
$B_{1}, B_{2,\alpha}$ and
$\widetilde{B_{2,-1}}$    are pairwise nonisomorphic.
\end{thm}

\begin{proof}
Firstly, $B_{1}$ is not isomorphic to $B_{2,\alpha}$ or
$\widetilde{B_{2,-1}}$, because their even parts are not isomorphic. According to Lemma \ref{lem:3.5}, the members of $B_{2,\alpha}$ are pairwise nonisomorphic. By Theorem \ref{thm:3.2}, we have that the even parts of  $\widetilde{B_{2,-1}}$ and $B_{2,\alpha}$ for $\alpha\neq \pm 1$ are not isomorphic.  Hence they are not isomorphic.

 Assume that $B_{2,-1}$ and $\widetilde{B_{2,-1}}$ are isomorphic. Let $\eta: B_{2,-1}\rightarrow \widetilde{B_{2,-1}}$ be an isomorphism and
  $Q=(q_{ij})$ be the matrix of $\eta$.  Since $\eta$  is an even linear map, we have that $Q$ is of the form
$$\left(\begin{array}{c|cc}\begin{smallmatrix}
q_{11}&q_{12}&q_{13}&q_{14}\\
q_{21}&q_{22}&q_{23}&q_{24}\\
q_{31}&q_{32}&q_{33}&q_{34}\\
q_{41}&q_{42}&q_{43}&q_{44}\end{smallmatrix}&0\\
\hline
0& \begin{smallmatrix}
q_{55}&q_{56}&q_{57}&q_{58}\\
q_{65}&q_{66}&q_{67}&q_{68}\\
q_{75}&q_{76}&q_{77}&q_{78}\\
q_{85}&q_{86}&q_{87}&q_{88}\\
 \end{smallmatrix}
 \end{array}\right).$$
 By $\eta(x_ix_j)=\eta(x_i)\eta(x_j), 1\leq i,j\leq 4$  and noting that the determinant of $Q$ cannot be zero,
we obtain  that  $$q_{11}=k,q_{22}=\frac{1}{k}, q_{33}=1, q_{44}=1, k\neq 0,k\in \mathbf{C} \textrm{ and other } q_{ij}, 1\leq i,j\leq 4, \textrm{ are zero}.$$
We proceed with an analogous computations for $\eta(x_1y_j)=\eta(x_1)\eta(y_j) (j=1,2,3,4)$.
Thus $$q_{55}=q_{56}= q_{57}=q_{58}=0,$$ which is a contradiction to the assumption that $\eta$ is an isomorphism.
Thus $B_{2,-1}$ and $\widetilde{B_{2,-1}}$ are not isomorphic.
Similarly, $B_{2,1}$ is not isomorphic to $\widetilde{B_{2,-1}}$. This completes the proof.
\end{proof}

We close this paper with the following question.

\begin{que}{\rm
 We would like to point out  the problem  that asks whether the  Lie superalgebras $A(m,n) (0\leq m<n)$ except for $A(0,1)$
have the  compatible LSSAs  and how to determine them remains open.
}
\end{que}

\section*{\textit{Acknowledgments}}

The author would like to thank Professor Jean-Louis
Loday and Yucai Su for their  discussion and useful suggestions.
This work was supported by NNSF of China (11226051) and the Fundamental Research
Funds for the Central Universities (11QNJJ001).


\begin{thebibliography}{9999}

\bibitem{AB2010} I. Ayadia and S. Benayadi, Symmetric Novikov superalgebras. J. Math. Phys. \textbf{51} (2010) 023501.

\bibitem{Bau1999} O. Baues, Left-symmetric algebras for $gl(n)$. Trans. Amer.
Math. Soc. \textbf{351} (1999) 2979-2996.

\bibitem{Bur1996}  D. Burde, Left-invariant affine structures on reductive Lie groups. J. Algebra
\textbf{181} (1996) 884-902.

\bibitem{Cha2004} F. Chapoton, Classification of some simple graded pre-Lie algebras of growth one. Comm. Algebra \textbf{32} (2004)
243-251.

\bibitem{Chu1974} B. Chu, Symplectic homogeneous spaces. Trans. Amer. Math. Soc. \textbf{197} (1974) 145-159.

\bibitem{DIO1997}  K. Dekimpe, P. Igodt and V. Ongenae, The five-dimensional complete left-symmetric algebra structures compatible with an abelian Lie algebra structure. Linear Algebra Appl. \textbf{263} (1997) 349-375.


\bibitem{Ger1963} M. Gerstenhaber,  The cohomology structure of an associative ring.  Ann. Math. \textbf{78} (1963) 267-288.


\bibitem{He1979} J. Helmstetter,  Radical d'une alg\`{e}bre sym\'{e}trique a gauche. Ann. Inst. Fourier \textbf{29} (1979) 17-35.

\bibitem{Kac1977} V. Kac, Lie superalgebras.  Adv.  Math. \textbf{26} (1977) 8-96.

\bibitem{KB2008}  X. Kong and C. Bai, Left-symmetric superalgebra structures on  the super-virasoro algebras. Pacific J. Math.
\textbf{235} (2008) 43-55.

\bibitem{KC2009}  Y. Kang and Z. Chen, Novikov superalgebras in low dimensions.  J. Nonlinear Math. Phys. \textbf{16} (2009) 251-257.

\bibitem{Kim1986} H. Kim,  Complete left-invariant affine structures on nilpotent Lie groups. J. Differ. Geom. \textbf{24} (1986)
373-394.


\bibitem{Kos1961} J. Koszul, Domaines born\'{e}s homog\`{e}nes et orbites de groupes de transformations affines. Bull. Soc. Math. France \textbf{89} (1961) 515-533.

\bibitem{Med1981} A. Medina Perea, Flat left-invariant connections adapted to the automorphism structure of a Lie group. J. Differ. Geom. \textbf{16} (1981) 445-474.


\bibitem{VM1996}  E. Vasilieva and A. Mikhalev, Free left-symmetric superalgebras. Fund. Appl. Math.  \textbf{2} (1996) 611-613.

\bibitem{WHB2010}  Y. Wang, D. Hou and C. Bai, Operator forms of the classical Yang-Baxter equation in
Lie superalgebras. Int. J. Geom. Methods Mod. Phys. \textbf{7} (2010) 583-597.


\bibitem{Xu2000a}  X. Xu,  Variational calculus of supervariables and related algebraic structures. J. Algebra \textbf{223} (2000) 396-437.

\bibitem{Xu2000b}  X. Xu, Quadratic conformal superalgebras. J. Algebra \textbf{231} (2000) 1-38.


\bibitem{Zha2011} R. Zhang, Left-symmetric superalgebras and some related superalgebrac structures. Ph.D thesis, Chern Inst. Math, Nankai University, Tianjin (2011).

\bibitem{ZB2012} R. Zhang and C. Bai, The classification of left-symmetric superalgebras in low dimensions. J. Algebra Appl. \textbf{11} (2012) 1250097 (26 pages).

\bibitem{ZHB2011} R. Zhang, D. Hou and C. Bai, A Hom-version of the affinizations of Balinskii-Novikov and Novikov superalgebras. J. Math. Phys. \textbf{52} (2011) 023505.

\end{thebibliography}
\end{document}